\documentclass[11pt,a4paper]{article}

\usepackage[OT1]{fontenc}
\usepackage[latin1]{inputenc}
\usepackage[english]{babel}
\usepackage{amsfonts}
\usepackage{amsthm}
\newtheorem{teo}{Theorem}[section]
\newtheorem{prop}[teo]{Proposition}
\newtheorem{lem}[teo]{Lemma}
\newtheorem{coro}[teo]{Corollary}

\newtheorem{rem}[teo]{Remark}

\def\h{{\cal H}}
\def\b{{\cal B}}
\def\bh{{\cal B}({\cal H})}

\def\a{\mathcal A}

\begin{document}

\title{\vspace*{0cm}The left invariant metric in the general linear group}
\date{}

\author{E. Andruchow, G. Larotonda, L. Recht, A. Varela}

\maketitle

\abstract{\footnotesize{\noindent Left invariant metrics induced by the $p$-norms of the trace in the matrix algebra are studied on the general lineal group. By means of the Euler-Lagrange equations, existence and uniqueness of extremal paths for the length functional are established, and regularity properties of these extremal paths are obtained. Minimizing paths in the group are shown to have a velocity with constant singular values and  multiplicity. In several special cases, these geodesic paths are computed explicitly. In particular the Riemannian geodesics, corresponding to the case $p=2$, are characterized as the product of two one-parameter groups.  It is also shown that geodesics are one-parameter groups if and only if the initial velocity is a normal matrix. These results are further extended to the context of compact operators with $p$-summable spectrum, where a differential equation for the spectral projections of the velocity vector of an extremal path is obtained.}}

\setlength{\parindent}{0cm} 

\section{Introduction}\label{intro1}

The object of this paper is to study the Finsler metric in the general linear group $G=GL(N)$ given by left translation of the $p$-norm induced by the trace in the space of $N\times N$ complex matrices. Namely, since $GL(N)$ is open in the space $M_N(\mathbb{C})$,  the tangent space of $GL(N)$ at any point identifies with  $M_N(\mathbb{C})$, and if $x\in M_N(\mathbb{C})$ is regarded as a tangent vector at $g\in GL(N)$, then the metric we consider is given by
$$
\|x\|_g=\|g^{-1}x\|_p=\tau((x^*(g^{-1})^*g^{-1}x)^n)^{1/p},
$$
where $p=2n$ is a fixed even integer, and $\tau$ is the {\it normalized real part} of the trace.

This metric is Riemannian when $p=2$, and as noted by V. I. Arnold \cite[Section 2]{arnold}, it is the natural metric in the study of the Lie group of movements of the generalized rigid body problem.  Thus  the geodesics of this group obey the least action principle for the kinetic energy of the body. In particular, one-parameter groups (which we will show that correspond to normal initial speeds in Section \ref{normales}) are the rotations of the rigid body around its major axes in the moment of inertia ellipsoid.

If $g(t)$, $T\in[0,1]$ is a smooth curve in $G$, the length functional is defined by
$$
{\cal L}_p(g)=\int_0^1 \|g^{-1}(t)\dot{g}(t)\|_p dt,
$$
and the distance between $g,h\in GL(N)$ is defined as the infima of the lengths of piecewise smooth curves joining $g,h$ in $GL(N)$.

Our goal is to characterize and establish the existence and uniqueness of smooth ($C^1,C^{\infty}$, etc.) minimizing paths for this metric, studying the Euler-Lagrange equation of the $p$-energy functional
$$
{\cal E}_p(g)=\int_0^1 \|g^{-1}(t)\dot{g}(t)\|_p^p dt,
$$
for $g(t)\in GL(N)$ a smooth curve parametrized in the interval $[0,1]$.

With the notable exception of the case $p=2$ (which corresponds to the Riemannian situation), the $p$-energy and the $p$-length functional are degenerate, in the sense that the second variation of each functional is nonnegative, but has directions of degeneracy where it vanishes. In fact, for $p>2$ and any initial condition $v\in M_n(\mathbb C)$, there are plenty of directions of degeneracy, unless $v$ is nonsingular (i.e. invertible). Thus this manifold lies beyond the reach of the standard techniques of Finsler geometry, and in particular the existence of  geodesic neighbourhoods is not guaranteed.

This paper is divided into two main parts, concerning matrix algebras and compact operators on a separable complex Hilbert space, and it is organized as follows. Suppose that a smooth curve $g(t)\in G$, $t\in[0,1]$ is given. Denote by
$$
v(t)=g^{-1}(t)\dot{g}(t)
$$
the left translation of the velocity vectors. In Section \ref{energia} it is shown that $g$ is an extremal of the $p$-energy functional, if and only if $v$ satisfies the differential equation
$$
\frac{d}{dt} v(v^*v)^{n-1}=(v^*v)^n-(vv^*)^n,
$$
which we call the {\it Euler-Lagrange equation} (of the $p$-metric). If $p>2$ or if $v$ is non-normal, this equation is difficult to deal with. Using the one-to-one and smooth map
$$
v\mapsto w:=v(v^*v)^{n-1},
$$
which we call the {\it Legendre transformation}, this equation becomes the {\it Hamilton} equation:
$$
\dot{w}=(w^*w)^{q/2}-(ww^*)^{q/2}=|w|^q-|w^*|^q,
$$
where $\frac1q+\frac1p=1$ and $|z|=\sqrt{z^*z}$ is the unique positive square root of the positive matrix $z^*z$. Solutions $w$ of this equation have the remarkable property that the eigenvalues and multiplicities of $w^*w$ (and of $ww^*$) are constant with respect to $t$. This fact implies that the Hamilton equation has (unique) solutions defined for all $t\in\mathbb{R}$. Then we reverse the Legendre transform to prove that solutions $v=v(t)$  exist for any initial data, they are unique and $C^{\infty}$. In Section \ref{minimizantes}, we establish some basic facts concerning the rectifiable distance $\ell_p$ for continuous maps, and its relation with the $p$-length functional: rectifiable arcs are almost everywhere differentiable, and minimal rectifiable arcs have constant speed (where it exists). Then, minimal curves for the rectifiable distance are minimal points of the $p$-energy (a simple consequence of H\"older's inequality) and thus fulfill the Euler-Lagrange equation, which forces them to be $C^{\infty}$. By means of Cohn-Vossen's theorem for locally compact geodesic length spaces, we prove that for given $g_1,g_2\in G$, there exists a unique $C^{\infty}$ curve of minimal length, such that the left translation of its velocity vectors verifies the Euler-Lagrange equation. Under certain special conditions of the initial speed, these curves are computed for any $p=2n$. In Section \ref{schatten}, we study the classical Banach-Lie groups whose Banach-Lie algebras are compact operators with $p$-summable singular values ($p$-Schatten operators), and we establish the smoothness ($C^1$ in this case) of the critical points of the rectifiable length, and obtain a differential equation of the spectral projections of these extremal curves.
\section{Matrix algebras}

In this section, we work in the algebra of $N\times N$ complex matrices ${\mathcal A}=M_N(\mathbb C)$, and we denote with $G=GL(N)$ the open subgroup of invertible matrices. Its tangent space at the identity (the Lie algebra) is identified with the full-matrix algebra, and we consider several norms induced by the standard, normalized ($\tau(1)=1$) real part of the trace of the algebra, $\tau:\mathcal A\to \mathbb R$. Consider
$$
\|x\|_p^p=\tau|x|^p\mbox{ for any }p\ge 1,
$$
which is called the {\it $p$-norm}. In what follows, $|x|=\sqrt{x^*x}$. The {\it uniform} norm  is given by $\|x\|=\|x\|_{\infty}=\max\{ \|x\xi\|: \, \xi \in\mathbb{C}^N,\, \|\xi\|\le 1\}$. It is also called the {\it supremum} norm, or the {\it spectral} norm. It is worth mentioning here, that for any $x\in {\mathcal A}$, and $1\le r\le s$, it holds that
$$
\|x\|_1\le \|x\|_r\le \|x\|_s\le \|x\|_{\infty}
$$
and also that the uniform norm can be arbitrarily approximated by the $p$-norms for $p$ sufficiently large.

\subsection{Euler-Lagrange equations and critical points}\label{energia}

We establish the existence and uniqueness of extremal paths for the $p$-energy functional. The computations are only carried out for even $p$. We postpone to Section \ref{minimizantes} the relationship with the $p$-length functional, and the study of minimizing properties of the extremal curves.

\subsubsection{Variational calculus for the $p$-norms}

Let $p\ge 2$ be an even integer and put $n=p/2$. Consider the $p$-{\it energy} functional
$$
{\cal E}_p(g)=\int_0^1 \|v(t)\|_p^p dt=\int_0^1 \tau(v(t)^*v(t))^n dt,
$$
for a smooth curve $g(t)$ ($0\le t\le 1$) in the group $G$, where we put $v(t)=g^{-1}(t)\dot{g}(t)$, which is the left translation to the identity of $G$, of the velocity vector of $g$.

This functional is invariant for the left action of $G$: if we put $h(t)=kg(t)$ for some $k\in G$, we have $\dot{h}(t)=k\dot{g}(t)$ and $h^{-1}(t)\dot{h}(t)=g^{-1}(t)\dot{g}(t)$.

We assume that $g=g_s$ depends smoothly upon a parameter $s\in (-\varepsilon,\varepsilon)$, and we will use the apostrophe $\,'$ to denote the derivative with respect to the $s$-parameter. Let $w=g^{-1}g'$ and $v=g^{-1}\dot{g}$. Now we compute $\frac{\partial}{\partial s}\tau(v^*v)^n$, we will use the formula
$$
v'=\dot{w}+[v,w],
$$
that follows interchanging derivatives and using that $dx^{-1}=-x^{-1}(dx)x^{-1}$. Thus
$$
(v')^*=\dot{w}^*+[w^*,v^*].
$$
We have
$$
(\tau(v^*v)^n)'=n\tau(v^*v)^{n-1}(v^*v)',
$$
so we compute
\begin{eqnarray}
(v^*v)'& =& (v')^*v+v^*v'=(\dot{w}^*+[w^*,v^*])v+v^*(\dot{w}+[v,w])\nonumber\\
&=& (\dot{w}^*v+v^*\dot{w}) +([w^*,v^*]v+v^*[v,w]).\nonumber
\end{eqnarray}
Finally,
$$
\dot{w}^*v+v^*\dot{w}=\frac{d}{dt}(w^*v+v^*w)-(w^*\dot{v}+\dot{v}^*w).
$$
Note that, since $\tau$ indicates the real part of the trace, then $\tau(x)=\tau(x^*)$. For the time being, we have three terms
\begin{enumerate}
 \item $2n\tau (v^*v)^{n-1}v^*[v,w]$
\item $-2n\tau (v^*v)^{n-1}w^*\dot{v}$
\item $2n\tau (v^*v)^{n-1} \frac{d}{dt}(w^*v)$.
\end{enumerate}
The last term can be rewritten as follows:
$$
3'. \quad p\tau(v^*v)^{n-1}\frac{d}{dt}w^*v=p\frac{d}{dt}\tau(v^*v)^{n-1}w^*v-p\tau(\frac{d}{dt}(v^*v)^{n-1} ) w^*v.
$$
If we put together the second term with this last expression, we get
$$
-p\frac{d}{dt}\tau(v^*v)^{n-1}w^*v-p\tau(\frac{d}{dt}(v^*v)^{n-1} ) w^*v=-p\tau(\frac{d}{dt}v(v^*v)^{n-1})w^*.
$$
Hence
$$
\frac1p \frac{\partial}{\partial s}\tau(v^*v)^n=\frac{d}{dt}\tau(v(v^*v)^{n-1})w^*-\tau(\frac{d}{dt}v(v^*v)^{n-1})w^*+\tau(v^*v)^n-(vv^*)^n)w^*.
$$
The last term comes from the bracket and its adjoint. Now
$$
\frac{\partial }{\partial s}{\cal E}_p(g_s)=\langle v(v^*v)^{n-1},w\rangle\bigg|_{0}^1-\int_0^1 \langle \frac{d}{dt}v(v^*v)^{n-1},w\rangle dt+\int_0^1\langle (v^*v)^n-(vv^*)^n,w\rangle dt,
$$
where $\langle,\rangle$ denotes the inner product of ${\mathcal A}$ induced by the (real part of) the trace
$$
\langle x,y\rangle =\tau(y^*x).
$$
This is the first variation formula for the functional ${\cal E}_p$. If we consider variations $w(t,s)$ with fixed endpoints $w(0,s)=w(1,s)=0$, then the condition for $g$ to be an extremal point is
\begin{equation}\label{variational}
\frac{d}{dt} v(v^*v)^{n-1}=(v^*v)^n-(vv^*)^n.
\end{equation}
This is the Euler-Lagrange equation of our variational problem. In order to deal with it, we shall rewrite it as follows. Let $\frac1q+\frac1p=1$, i.e. $q$ is the conjugate exponent of $p$ (recall that $p\ge 2$ is an even integer). Note that $\frac{2n}{2n-1}=\frac{p}{p-1}=q$.

Consider the Legendre transformation given by
\begin{equation}\label{legendre transformation}
v\mapsto w=v(v^*v)^{n-1}
\end{equation}
and note that
\begin{equation}\label{modu}
(v^*v)^n=(w^*w)^{\frac{n}{2n-1}}\hbox{ and }  (vv^*)^n=(ww^*)^{\frac{n}{2n-1}}.
\end{equation}
Then the Hamilton equation  is
\begin{equation}\label{euler}
\dot{w}=|w|^q-|w^*|^q,
\end{equation}
where $|w|=\sqrt{w^*w}$ and $|w^*|=\sqrt{ww^*}$. In particular, $\dot{w}$ is self-adjoint, or stated in other way, the skew-adjoint part of $w$ must be constant.

\bigskip

In order to establish existence and uniqueness of solutions of the Hamilton equation, it suffices to show  that the map $a \mapsto |a|^q$ is locally Lipschitz. Since we are interested in solving this equation in two contexts, we recall here a result which covers both. The proof is based on results in \cite{davies} and  \cite{ps} (see also \cite{delasalle}).

\begin{rem}\label{lipschitz}
Let $1<r<\infty$, $a,b\in\mathcal A$ where  $\mathcal A$ is either a matrix algebra or the (unitized) ideal of compact $p$-Schatten operators (see Section \ref{schatten}). Then there exists a constant $c(r,d)>0$, such that if $\|a\|_r,\|b\|_r\le d$, then
$$
\||a|^r-|b|^r\|_r\le c(r,d)\|a-b\|_r.
$$
Indeed, E.B. Davies proved in \cite{davies} that if $a,b$ are  operators such that $a-b$ belongs to the $r$-Schatten class, then $|a|-|b|$ also belongs to this class, and there exists a constant $c_1(r)$ such that
$$
\||a|-|b|\|_r\le c_1(r)\|a-b\|_r.
$$
On the other hand, D. Potapov and F. Sukochev proved in \cite{ps} (see also \cite{delasalle}), that if $f:\mathbb{R}\to\mathbb{R}$ is a Lipschitz map, and $x,y$ are self-adjoint operators in the $r$-Schatten class, then $f(x)-f(y)$ belongs to the $r$-Schatten class, and there exists a constant $c_2(r)$ such that
$$
\|f(x)-f(y)\|_r\le c_2(r)\|x-y\|_r.
$$
Our assertion follows considering the function $f_d:\mathbb{R}\to\mathbb{R}$ given by
$$
f_d(t)=\left\{ \begin{array}{l} |t|^r \hbox{ if } |t|\le d \\
d \frac{t}{|t|} \hbox{ if } |t|>d. \end{array} \right.
$$
Clearly $f_d$ is a Lipschitz map, and if $x\in \b(\h)$ verifies that $\|x\|_r\le d$, $f_d(|x|)=|x|^r$. Our claim therefore follows.
\end{rem}

Note that, while in the matrix algebra $\mathcal A=M_N(\mathbb C)$ we are dealing with the topology induced by the uniform norm (all norms being equivalent for fixed $N$), one obtains that $a\mapsto |a|^q$ is locally Lipschitz for the uniform norm, with a constant depending only on $p,N$ and the radius of the ball where one wishes to obtain the estimate.

By the remark above, the Hamilton equation (\ref{euler}) has a continuously differentiable local solution $w:(t_-,t_+)\to {\mathcal A}$ for any initial condition $w(0)\in {\mathcal A}$ by the standard result for ODE's on Banach spaces (see for instance \cite[Chapter IV]{lang}). Moreover, the flow of solutions is a Lipschitz map with respect to the initial condition $w_0$.

\subsubsection{Special case: normal speed}\label{normales}

The equation (\ref{euler}) shows that the extremals of the variational problems are the one-parameter groups $g(t)=e^{tv}$ ($v\in {\mathcal A}$) if and only if $v$ is a normal element of ${\mathcal A}$. Indeed, if $g(t)=e^{tv}$, then $v(t)=g^{-1}(t)\dot{g}(t)\equiv v$, thus $g$ is an extremal if and only if
$$
(vv^*)^n=(v^*v)^n,
$$
which is  equivalent to $vv^*=v^*v$. Note that the set of such speeds in the algebra $\mathcal A$ is closed, but not a submanifold, since $f(x)=x^*x-xx^*$ has vanishing differential at $x=1$.

\subsubsection{Special case: partial isometries as speeds}

If the initial condition $v_0$ is a partial isometry, that is $v_0v_0^*$ and $v_0^*v_0$ are projections, then for any $p=2n$,
\begin{equation}\label{solu}
g(t)=g_0e^{tv_0^*}e^{t(v_0-v_0^*)},
\end{equation}
is the unique solution of the variational problem with $g(0)=g_0$, $\dot{g}(0)=g_0v_0$. Indeed, let $v_0=x_0+iy_0$ with $x_0,y_0$ self-adjoint, and define
$$
v(t)=g^{-1}\dot{g}(t)=e^{-t(v_0-v_0^*)}v_0e^{t(v_0-v_0^*)}=e^{-2ity_0}v_0e^{2ity_0}=e^{-2ity_0}x_0e^{2ity_0}+iy_0.
$$
Then, since $v$ is a curve of partial isometries,  $(v^*v)^n=v^*v$, $(vv^*)^n=vv^*$ and $v(v^*v)^{n-1}=v$. Hence a direct computation shows that
$$
\dot{v}=e^{-2ity_0}[-2iy_0,x_0]e^{2ity_0}=v^*v-vv^*
$$
which implies that $v$ is a solution of the Euler-Lagrange equation (\ref{variational}).

Note that the solution can be rewritten as
$$
g(t)=g_0e^{t(x_0-iy_0)}e^{2tiy_0}.
$$

\subsubsection{The spectrum of the velocity vector}

Returning to the general situation where $v_0$ is not a normal operator nor a partial isometry, note that
\begin{eqnarray}
 (w^*w)\dot{ }&=&(|w|^q-|w^*|^q)w+w^*(|w|^q-|w^*|^q)\nonumber\\
&=& |w|^qw-w|w|^q+w^*|w|^q-|w|^qw^*=2\ Sym[|w|^q,w],\nonumber
\end{eqnarray}
where $Sym(x)=\frac12(x+x^*)$.
Let $k\in {\mathcal A}$ be the skew-adjoint part of $w$, which as noted above, is constant. Then
$$
Sym[|w|^q,w]= [|w|^q,k].
$$
So, if we put $a=|w|$,
$$
(a^2)\dot{ }=2[a^q,k].
$$
Fix an element $K^*=-K\in\a$. The following auxiliary equation will be useful:
\begin{equation} \label{ecuacionauxiliar}
\dot{b}=[b^{\alpha},K]=b^{\alpha}K-Kb^{\alpha},
\end{equation}
where $\alpha=q/2$. Note that if $w$ is a solution of the Hamilton equation, then $b=|w|^2$ is a solution of (\ref{ecuacionauxiliar}) for $K=2k$.

\begin{lem}
If a continuously differentiable curve $b(t)$ of positive elements of $\a$ is a solution of equation (\ref{ecuacionauxiliar}) and $s\ge 1$, then $\|b(t)\|_s$ is constant and the eigenvalues and multiplicities of $b(t)$ do not depend on $t$. In particular, $\|b(t)\|$ is also constant.
\end{lem}
\begin{proof}
For each $n\in \mathbb N_0$,
$$
\frac{d}{dt}\tau(b^n)=n\tau(\dot{b}b^{n-1})=n\tau(b^{\alpha}Kb^{n-1}-Kb^{\alpha}b^{n-1})=0,
$$
Therefore, if $p$ is a polynomial, $\tau(p(b(t)))$ is constant. It follows that if $f$ is any continuous function on the real line, then $\tau(f(b(t)))$ is constant. Hence, for any $s>1$, by considering $f(t)=|t|^s$, we obtain
$$
\|b(t)\|_s^s=\tau|b(t)|^s=\tau (f(b(t)))=\tau(f(b(0)))=\|b(0)\|_s
$$
for any $t\in \mathbb R$ such the solution $b$ exists.

Let us show now that the spectrum of $b(t)$ is constant. Indeed, suppose otherwise that $\lambda\in\sigma(b(t_0))$ and $\lambda \notin \sigma(b(t_1))$. Let $f$ be a continuous function in the real line such that $0\le f(t)\le 1$,  $f(\lambda)=1$ and $f$ is zero in the spectrum of $b(t_1)$. It follows that
$$
0=\tau(f(b(t_1)))=\tau(f(b(t_0))).
$$
Since $f(b(t_0))\ge 0$, this implies that $f(b(t_0))=0$. On the other hand, $\|f(b(t_0))\|=f(\lambda)=1$. Thus the spectrum is constant, and if we recall that the spectral norm of a normal (in particular, positive) matrix $x$ can be computed as the maximum of $\lambda_i$, where $\lambda_i$ are the eigenvalues of $|x|=\sqrt{x^*x}$, then clearly $b(t)$ has constant uniform norm.
\end{proof}

\begin{rem}\label{extensiondelremark}
The same conclusion follows, if one supposes that $b(t)$ is a Lipschitz map (in particular, continuous and almost everywhere differentiable) and an a.e.-solution of equation (\ref{ecuacionauxiliar}). Indeed, reasoning as above, one has that $b^n(t)$ is a.e.-differentiable, and therefore $\tau(b^n(t))$ is constant for any $n\ge 0$.
\end{rem}

\begin{teo}\label{existenciatransformed}
The Hamilton equation (\ref{euler})
$$
\dot{w}=|w|^q-|w^*|^q , \ \ w(0)=w_0
$$
 has a unique continuously differentiable solution defined for all $t\in\mathbb{R}$.
 \end{teo}
 \begin{proof}
 Let $w(t)$ be a local solution for this problem, defined for $t\in (t_1,t_2)$. Then, by the computations leading to equation (\ref{ecuacionauxiliar}), it follows that $b(t)=w^*(t)w(t)$ is a continuously differentiable solution of  equation (\ref{ecuacionauxiliar}). By the above lemma, this implies that
 $$
 \|w(t)\|_q^q=\tau(|w(t)|^q)=\tau(|b(t)|^{q/2})=\|b(t)\|_{q/2}^{q/2}
 $$
 is constant. Note that in the inequality of Remark \ref{lipschitz},  the Lipschitz constant of the map $a\mapsto |a|^q$, depends on the $q$-norm of the initial condition, and careful inspection of the proof of the theorem of existence and uniqueness (see \cite[page 66]{lang}) shows that $t_1,t_2$ also depend only on it. Therefore, if we denote $w_1=w(t_1/2)$ and pose the Hamilton equation (\ref{euler}) with initial condition $w(t_1/2)=w_1$, this solution is defined on an interval of the same length as the previous solution, because $\|w(0)\|_q=\|w(t_1/2)\|_q$. Iterating this procedure, on both sides of the origin, one obtains a (unique) solution defined for all $t\in\mathbb{R}$.
\end{proof}

\subsubsection{Reversing the Legendre transform}

We would like to obtain solutions, defined for all time, of the original Euler-Lagrange equation. The problem here is that the Legendre transformation
$$
v \mapsto w=v(v^*v)^{n-1}
$$
is one-to-one and continuously differentiable, but not in general a diffeomorphism. Here, in the finite dimensional setting, we could argue using invariance of domain and thus obtaining that $t\mapsto v(t)$ is a continuous function, thus the differential equations $\dot{\gamma}=\gamma v$ will have a unique $C^1$ solution $\gamma$ for any initial data. But we can do better, we can explicitly reverse the transform. In order to do it, the following remark will be useful. Let us denote with $R(v)\subset \mathbb C^N$ the range subspace of $v\in\mathcal A$.

\begin{rem}\label{descompolar}
With notations as in the previous section, let
$$
w=\Omega |w|
$$ be the polar decomposition of $w$, $\Omega$ the unique partial isometry from $R(|w|)$ to $R(w)$ with kernel equal to $R(|w|)^{\perp}=N(|w|)$. Then the polar decomposition of $v$ is
$$
v=\Omega |w|^{\frac{1}{p-1}}=\Omega |w|^{\frac{1}{2n-1}}.
$$
Indeed,
$$
|w|=(w^*w)^{1/2}=((v^*v)^{2n-1})^{1/2}=|v|^{2n-1}=|v|^{p-1}
$$
and clearly $N(v)=N(|v|)=N(|v|^{2n-1})=N(w)$. Analogously,
$$
|w^*|=(ww^*)^{1/2}=(v(v^*v)^{2n-2}v^*)^{1/2}=((vv^*)^{2n-1})^{1/2}=|v^*|^{2n-1}=|v^*|^{p-1},
$$
and thus $R(v)=N(v^*)^\perp=N(|v^*|^{2n-1})^\perp=N(w^*)^\perp=R(w)$. Therefore the claim follows by the uniqueness property of the polar decomposition.
\end{rem}

As above, let $v=\Omega |v|$ be the polar decomposition of the solution  $v$, and $p_0,p_0^{\perp}$ stand for the projections to the kernel and range of $|v|$.

\begin{teo}\label{closedrange}
Each of the curves $|w|, p_0, \Omega (=\Omega p_0^{\perp})$ and $v$ are $C^{\infty}$ maps. In particular, the Euler-Lagrange equation (\ref{variational}) has a unique continuously differentiable solution $v(t)$ for $t\in\mathbb{R}$, with $v(0)=v_0$, which is in fact $C^{\infty}$.
\end{teo}
\begin{proof}
Let $b_0=|w_0|^2$, and recall that $\sigma(b(t))$ is constant. Now if $f(\lambda)=\sqrt{\lambda}$ is the principal branch of the complex square root, then
$$
|w(t)|=\frac{1}{2\pi i}\oint_C \sqrt{\lambda} (\lambda- b(t))^{-1}d\lambda
$$
which shows that $t\mapsto |w(t)|$ is $C^1$, since $b(t)=w^*w(t)$, and $w$ is $C^1$ by Theorem \ref{existenciatransformed} (we picked any simple smooth positively oriented path $C$ around the non zero part of the spectrum of $b(t)$). Analogously, if $\lambda^q=\exp(q\log(\lambda))$ denotes the principal branch of the $q$-power, then
$$
|w(t)|^q=\frac{1}{2\pi i}\oint_C \lambda^q (\lambda- |w(t)|)^{-1}d\lambda,
$$
where the curve $C$ now is taken around the non zero part spectrum of $|w(t)|$. This shows that $|w(t)|^q$ is continuously differentiable, and with the same arguments, $|w(t)^*|, |w(t)^*|^q$, and $|w(t)|^{\frac{1}{p-1}}$ are continuously differentiable. Then
$$
\dot{w}=|w|^q-|w^*|^q
$$
is $C^1$, i.e. $w$ is a $C^2$ map. Iterating this argument, it follows that $w$ is a $C^{\infty}$ map, and the same holds for $|w|$ and all its powers. Since
$$
p_0(t)^{\perp}=1-p_0(t)=\frac{1}{2\pi i}\oint_C  (\lambda- b(t))^{-1}d\lambda
$$
is the projection to the range of $|w(t)|$, the map $t\mapsto p_0(t)$ is $C^{\infty}$ and the same applies to $t\mapsto p_0(t)^{\perp}$. Then $|w(t)|+p_0(t)$ is  invertible and $C^{\infty}$. Let
$$
\mu(t)=(|w(t)|+p_0(t))^{-1}p_0(t)^{\perp},
$$
then $\mu$ is a $C^{\infty}$ map such that
$$
|w|\mu=(|w|+p_0-p_0)(|w|+p_0)^{-1}p_0^{\perp}=p_0^{\perp}-0=p_0^{\perp},
$$
hence
$$
\Omega p_0^{\perp}=\Omega|w|\mu=w\mu,
$$
which is $C^{\infty}$ (where $w=\Omega|w|$ is the polar decomposition of $w$ as before). Then
$$
v=\Omega |w|^{\frac{1}{p-1}}=\Omega p_0^{\perp} |w|^{\frac{1}{p-1}}=w\mu |w|^{\frac{1}{p-1}}
$$
is $C^{\infty}$, and as noted before, the solution of equation (\ref{variational}).
\end{proof}

\subsection{Evolution of the initial speed of an extremal curve}

Since the spectra of $|v|$ and $|w|$ are constant and finite, one can describe these solutions by means of the action of the unitary group, on positive matrices and on partial isometries. First we recall some basic facts.

\begin{rem}\label{equivariant}
\begin{enumerate}

\noindent
\item
Let $\{\lambda_1,\dots,\lambda_k\}$ be the spectrum of $|v(0)|\setminus\{0\}$ (if $0$ belongs to the spectrum of $|v(0)|$, we denote it by $\lambda_0$). Put $p_i(t)$ the spectral projection of $|v(t)|$ corresponding to $\lambda_i$. Then there exists a $C^\infty$ curve of unitaries $u(t)$ such that $p_i(t)=u(t)p_i(0)u^*(t)$. There are many ways to construct $u(t)$. For instance,  with the same argument involving Riesz integrals as above, it can be shown that the curves $p_i(t)$ are $C^\infty$. Thus one can define a $C^\infty$ curve of matrices
$$\Lambda_t=-\sum_{j=0}^k p_j\dot{p}_j.
$$
Note that differentiating $p_i^2(t)=p_i(t)$, one obtains $\dot{p}_ip_i+p_i\dot{p}_i=\dot{p}_i$. Also $\dot{p}_i^*=\dot{p}_i$. Combining these and the fact that $1=\sum_{j=0}^k p_i$, one obtains that $\Lambda_t$ is anti-Hermitian:
$$
\Lambda_t^*=-\sum_{j=0}^k \dot{p}_jp_j=-\Lambda.
$$
Consider the linear differential equation in $M_N(\mathbb{C})$:
$$
\left\{ \begin{array}{l} \dot{u}(t)=\Lambda_t u(t) \\ u(0)=1 \end{array}\right. .
$$
Apparently the unique solution is a curve of unitaries. Moreover,
$$
(u^*p_iu)^\cdot=u^*\Lambda p_i+u^*\dot{p}_iu+u^*p_i\Lambda u=u^*\{-\dot{p}_ip_i+\dot{p}_i-p_i\dot{p}_ip_i\}u=0,
$$
by the above identity (here we use that $p_i\Lambda=-p_i\dot{p}_i$ and $\Lambda^* p_i=-\dot{p}_ip_i$). Thus,
$$
u^*(t)p_i(t)u(t)=p_i(0),
$$
and it follows that
$$
|v(t)|=\sum_{j=1}^k\lambda_jp_j(t)=u(t)|v(0)|u^*(t).
$$
\item
Accordingly, there exists a $C^\infty$ curve of unitary matrices $\nu(t)$, such that if $v=\Omega|v|$ is the polar decomposition of the curve $v$, then
$$
\Omega(t)=\nu(t)p_0^\perp(t).
$$
To prove it, denote by $\cal{I}$ the set of partial isometries, and by $\cal{P}$ the set of projections. The unitary group $U_N(\mathbb{C})$, acts on both  sets, by means of the actions:
$$
(u_1,u_2)\cdot \Omega=u_1\Omega u_2^*\ , \ \ u\cdot p=upu^*,
$$
for $\Omega\in\cal{I}$, $p\in\cal{P}$, $u_1,u_2,u\in U_N(\mathbb{C})$. The orbits of these actions are connected components of $\cal{I}$ and $\cal{P}$. In particular, the connected components of $\cal{I}$ and $\cal{P}$ are $C^\infty$ submanifolds of $M_N(\mathbb{C})$. Note that the curve of partial isometries $\Omega(t)$ of $v(t)$ lies in one connected component, as does the curve of projections $p_0^\perp(t)$. Denote from now on  by $\cal{I}$ and $\cal{P}$ precisely these components. Consider the map
$$
U_N(\mathbb{C})\times {\cal P}\to {\cal I} ,\quad (u,p)\mapsto up,
$$
which is clearly $C^\infty$. We claim that it is a submersion. Indeed, in \cite{acm} it was shown that if $v,v_0\in \cal{I}$,  $p_0=v_0^*v_0$ is the initial projection of $v_0$, and $\|v-v_0\|_\infty<1/2$, then there exist unitaries $u_1$ and $u_2$ (which are $C^\infty$-maps in the variables $v$ and $v_0$) such that $v=u_1p_0u_2^*$. Then it is easy to check that $v\mapsto (u_1u_2^*, u_2p_0u_2^*)$ is a $C^\infty$-cross section for the map $(u,p)\mapsto up$, defined in a neighbourhood of $v_0$. Therefore, the $C^\infty$ curve $\Omega(t)$ can be lifted to a  pair of $C^\infty$ curves $(\nu(t),p(t))$ in $U_N(\mathbb{C})\times \cal{P}$ such that $\Omega(t)=\nu(t)p(t)$,  where $p(t)=\Omega^*(t)\Omega(t)$ is the curve of initial projections, i.e. $p(t)=p_0^\perp(t)$ in our current notations.
\item
Putting these facts together, we obtain that the solution $v(t)$ of the Euler-Lagrange equation (\ref{variational}) can be written as
$$
v(t)=\nu(t)u(t)v(0)u^*(t).
$$
where $\nu,u$ are $C^{\infty}$ paths of unitary matrices with $\nu(0)=u(0)=1$. In particular, it follows that not only the spectrum of $|v|$ is constant, but also its multiplicity (i.e. the multiplicity of each eigenvalue).
\end{enumerate}
\end{rem}

\subsection{The length functional and minimal paths}\label{minimizantes}

Let ${\mathcal L}_p$ denote the $p$-length of piecewise $C^1$-paths in $G$,
$$
{\mathcal L}_p(\alpha)=\int_0^1 \|\alpha^{-1}\dot{\alpha}\|_p,
$$
and define the rectifiable distance as the infima of such paths joining given endpoints,
$$
d_p(h,k)=\inf\{ {\mathcal L}_p(\alpha):\alpha(0)=h,\alpha(1)=k\}.
$$
Recall the definition of the rectifiable metric,
$$
\ell_p(\alpha)=\sup_{\pi} \sum_i d_p(\alpha(t_i),\alpha(t_{i+1})),
$$
where the supremum is taken over all partitions $\pi$ of the interval $[0,1]$.

We say that path $\alpha$ is {\it rectifiable}, if it is continuous and $\ell_p(\alpha)<\infty$. We define the rectifiable distance as the infima of rectifiable paths joining given endpoints,
$$
\overline{d}_p(h,k)=\inf\{ \ell_p(\alpha):\alpha(0)=h,\alpha(1)=k\}.
$$

\subsubsection{Rectifiable paths}

By a standard argument (that we omit, see for instance \cite[Chapter 1]{gromov}), if $\alpha$ is a piecewise $C^1$ path, then $\ell_p(\alpha)\le {\mathcal L}_p(\alpha)$, and moreover the metric space $G$ is an inner metric space, that is $\overline{d}_p=d_p$.

Since $\ell_p$ and ${\mathcal L}_p$ are invariant under re-parametrization, in this context we can assume that any rectifiable curve $\alpha:[0,1]\to G$ is Lipschitz continuous, parametrized with constant speed. That is,
$$
d_p(\alpha(t),\alpha(s))\le \ell_p(\alpha)|t-s|
$$
for any $t,s\in [0,1]$, and moreover
\begin{equation}\label{constante}
\ell(\alpha|_{[t,t+s]})=s\ell(\alpha).
\end{equation}

\begin{rem} Since $\alpha$ is Lipschitz,
$$
v_{\alpha}(t)=\lim_{h\to 0} \frac{d_p(\alpha(t+h),\alpha(t))}{|h|}
$$
exists almost everywhere, and moreover, the Lebesgue integral $\int_0^1 v_{\alpha}(t)dt$ exists and equals $\ell_p(\alpha)$. See \cite[Theorem 2.7.6]{burago} for a proof.
\end{rem}

We have
$$
\|\alpha(t)-\alpha(s)\|_p\le d_p(\alpha(t),\alpha(s))\le \ell_p(\alpha)|t-s|
$$
for any $t,s\in [0,1]$. Then (see \cite[Section 8.1]{ambrosio}) the usual derivative
$$
\dot{\alpha}(t)=\lim\limits_{h\to 0}\frac{\alpha(t+h)-\alpha(t)}{h}
$$
exists almost everywhere.

\begin{rem}
We remark here that in the case of compact operators considered in Section \ref{schatten} below, for $1<p<\infty$, the limit can be taken in the norm topology induced by the $p$-norm, since this is a uniformly convex space. See \cite[Proposition III.30]{brezis} for the details.
\end{rem}

From here it follows easily that for any for any $s,t\in [0,1]$,
$$
\alpha(t)-\alpha(s)=\int_s^t \dot{\alpha}(h)dh,
$$
and moreover the Lebesgue integral
$$
{\mathcal L}_p(\alpha)=\int_0^1 \|\alpha^{-1}\dot{\alpha}\|_p
$$
is well-defined, and the same holds true for the energy functional. Note also that, due to Hölder's inequality, for any rectifiable path one has
\begin{equation}\label{holder}
{\mathcal L}_p(\alpha)^p=(\int_0^1\|\alpha^{-1}(t)\dot{\alpha}(t)\|_p dt)^p\le \int_0^1\|\alpha^{-1}(t)\dot{\alpha}(t)\|_p^p dt={\cal E}_p(\alpha).
\end{equation}

\begin{rem}\label{esconstante}
Assume that $\alpha$ is a rectifiable and minimizing curve for the length functional (since $\overline{d}_p=d_p$, there is no ambiguity here). We assume that $\alpha$ is parametrized with constant speed. Then it is easy to check that $\alpha$ is also minimizing for both functionals, on any subinterval $[s,s+h]\subset [0,1]$. Thus by equation (\ref{constante}),
$$
\frac1s\int_s^{s+h} \|\dot{\alpha}(t)\|_{\alpha(s)}=\frac1s d_p(\alpha(s),\alpha(s+h))=\frac1s s\ell_p(\alpha),
$$
where again $\int$ denotes the Lebesgue integral. It follows that $\|\alpha^{-1}\dot{\alpha}(t)\|_p=constant=\ell_p(\alpha)$ for any $t\in [0,1]$ where the derivative exists.
\end{rem}

\subsubsection{Length and energy}

When considering the relation between minima and critical points of the length and energy functionals,
$$
{\mathcal L}_p(\gamma)=\int_0^1\|\gamma^{-1}(t)\dot{\gamma}(t)\|_p dt \ , \  \ {\cal E}_p(\gamma)=\int_0^1\|\gamma^{-1}(t)\dot{\gamma}(t)\|^p_p dt,
$$
it is important to note that the length functional is invariant under re-parametrizations, while the energy functional is not.

\begin{prop}\label{escritica}
Let $\gamma(t)$ be a curve in $G$, which is a Lipschitz map, $t\in[0,1]$.
\begin{enumerate}
\item
If $\gamma$ is a minimum of the length functional ${\mathcal L}_p$, then its re-parametrization by arc-length is a minimum of ${\cal E}_p$.
\item
If $\gamma$ is a critical point of  ${\cal E}_p$, then it is a critical point of ${\mathcal L}_p$.
\end{enumerate}
\end{prop}
\begin{proof}
Let $\gamma$ be a minimum of ${\mathcal L}_p$, the same holds true for its re-parametrization by arc-length (which we still call $\gamma$). By Remark \ref{esconstante},  $\|\gamma^{-1}\dot{\gamma}\|_p=c\; \; \; a.e.$, thus
$$
{\mathcal L}_p(\gamma)^p=\left(\int_0^1\|\gamma^{-1}(t)\dot{\gamma}(t)\|_p dt\right)^p=c^p={\cal E}_p(\gamma).
$$
If $\alpha$ is any other Lipschitz curve in $G$, by H\"older's inequality (\ref{holder}),
$$
{\cal E}_p(\gamma)={\mathcal L}_p(\gamma)^p\le {\mathcal L}_p(\alpha)^p\le {\cal E}_p(\alpha),
$$
which proves the first claim.

Suppose now that $\gamma$ is a critical point of ${\cal E}_p$. Consider the energy functional with its Lagrangian
$$
{\mathcal E}(\alpha)=\int_0^1 E(\alpha,\dot{\alpha}) dt,
$$
where $E(u,z)=\|u^{-1}z\|_p^p=Tr[(u^{-1}z)^*u^{-1}z]^n$, $E:G\times {\mathcal A}\to \mathbb R$ is a $C^1$ map.
 Then $\gamma$ is a weak-Lipschitz solution of the Euler-Lagrange equation, with the same proof as in \cite{giaquinta}, Remark 2 in page 40 and Proposition 2 in page 41, loc. cit. That is, $\frac{d}{dz}E(\gamma(t),\dot{\gamma}(t))$ is absolutely continuous and verifies the Euler-Lagrange equation
$$
\frac{d}{dt}\frac{d}{dz}E(\gamma(t),\dot{\gamma}(t))=\frac{d}{du}E(\gamma(t),\dot{\gamma}(t))
$$
almost everywhere on $[0,1]$.   Calling ${\nu}={\gamma}^{-1}\dot{\gamma}$, in our particular situation, we know that the Euler-Lagrange equation reduces to
$$
\frac{d}{dt}\nu(\nu^*\nu)^{n-1}=(\nu^*\nu)^n-(\nu\nu^*)^n.
$$
Then $\beta=\nu^*\nu$ is a Lipschitz map and an almost everywhere solution of the equation (\ref{ecuacionauxiliar}). By Remark \ref{extensiondelremark}, this implies that $\beta(t)$ has constant spectrum. More precisely,
$$
\|\nu(t)\|_p=\||\nu(t)|\|_p=\|\beta(t)^{1/2}\|_p
$$
is constant. Let $\gamma_s(t)$ be a variation of $\gamma_0=\gamma$ (i.e. for each  $s\in (-r,r)$, $\gamma_s$ is a Lipschitz map with values in $G$, and it is differentiable with respect to the parameter $s$). Put $\nu_s=\gamma_s^{-1}\dot{\gamma}_s$. Then
$$
\frac{d}{d s}|_{s=0} {\mathcal L}_p(\gamma_s)=\int_0^1 \frac{d}{d s}|_{s=0}\|\nu_s\|_p d t.
$$
Note that
$$
p\frac{d}{d s}\|\nu_s\|_p=p\frac{d}{d s}(\|\nu_s\|^p_p)^{1/p}=\|\nu_s\|^{1/p-1}\frac{d}{d s}\|\nu_s\|^p_p.
$$
At $s=0$, $\|\nu_0(t)\|_p=\|\nu\|_p=c$ is constant. Thus
$$
p\frac{d}{d s}\Big|_{s=0} {\mathcal L}_p(\gamma_s)=c^{1/p-1}\int_0^1 \frac{d}{d s}\Big|_{s=0}\|\nu_s\|^p_p=c^{1/p-1}\frac{d}{d s}\Big|_{s=0}{\cal E}_p(\gamma_s)=0.
$$
\end{proof}

\begin{rem}\label{esabso}
Let us emphasize, in the above proof, the fact that the partial derivative
$$
\frac{d}{dz}E(\gamma(t),\dot{\gamma}(t))
$$
is absolutely continuous, when $\gamma$ is a rectifiable minimizer. In our context this means that
$$
[0,1]\ni t\mapsto \nu(\nu^*\nu)^{n-1}=\omega
$$
is absolutely continuous. Then, $\omega^*\omega$ is also continuous, and it has constant spectrum as we mentioned earlier, and the same applies to $|\omega|$, since
$$
\dot{w}=|w|^q-|w^*|^q.
$$
It should be noted that the continuity of $|\omega|$ implies that $\omega$ is in fact a $C^1$ map.
\end{rem}

Combining the last remark with Theorem \ref{closedrange}, we obtain the following characterization of minimal rectifiable arcs in $G$.

\begin{coro}\label{suaave}
If $\gamma$ is a rectifiable and minimizing curve for the $p$-distance, then $\gamma$ is $C^{\infty}$ in ${\mathcal A}$, and the unique solution of the Euler-Lagrange equation, for given initial conditions.
\end{coro}

\subsection{The Riemannian case}

Choosing $p=2$, we find ourselves in the realm of Riemannian geometry. Left-invariant metrics on Lie groups have been extensively discussed; for instance we refer the reader to the beautiful appendix in the book by V. I. Arnol'd on classical mechanics \cite[Appendix 2]{arnold}. Lie groups considered by Arnol'd  are real, but since we are working with the real part of the trace, his observations can be applied to our context with some caution. It is shown there that these equations of motion correspond to the case of a (generalized) rigid body. In particular, the case of one-parameter groups (which, as we noted, correspond to normal initial speeds) are the rotations of the rigid body around its major axes of the moment of inertia ellipsoid. Some of the terms and remarks in this section are related to Arnol'd's exposition.

Note that the Euler-Lagrange equation  (\ref{variational}) of the variational problem becomes
$$
\dot{v}=v^*v-vv^*.
$$
The solutions in this case are locally minimizing among piecewise smooth curves joining the same endpoints, by the standard argument of Riemannian geometry.

What is remarkable here (and to the best to our knowledge seems to be new), is that the solutions can be computed explicitly for any initial position $g_0\in G$ and any initial speed $g_0v_0\in \mathcal A$.

\begin{teo}\label{gilberto}
Let $g_0\in G$, $v_0\in \mathcal A$. Then the unique geodesic of the Levi-Civita connection induced by the trace inner-product metric on the invertible group $G$ of $\mathcal A$, with initial position $g_0$ and initial speed $g_0v_0$, is given by
$$
g(t)=g_0e^{tv_0^*}e^{t(v_0-v_0^*)}.
$$
\end{teo}
\begin{proof}
Recall the auxiliary equation (\ref{ecuacionauxiliar}) of the positive part $b=|w|^2$ given by
$$
\dot{b}=b2k-2kb,
$$
and note that here $v=w$, that is, the Legendre transformation is the identity map. Recall also that the skew-adjoint part of $v=v(t)$ is constant, hence $v(t)=h(t)+ik$ with $h(t)$ smooth and self-adjoint. Note that $2bk-2kb=(R_{2k}-L_{2k})(b)$, where $R$ and $L$ denote right and left multiplication respectively; this yields
$$
b(t)=e^{2tR_k}e^{2tL_{-k}}b_0= e^{-2itk}b_0e^{2itk}.
$$
On the other hand, the proposed solution $g(t)=g_0e^{t(h-ik)}e^{2itk}$ readily verifies the same equation.
\end{proof}

In this context, the Riemannian exponential map is given, for fixed $g\in GL(N)$, by the expression
$$
Exp(v)=ge^{v^*}e^{v-v^*},
$$
and the exponential flow is certainly a smooth ($C^{\infty}$) map from $\mathbb R\times \mathcal A$ to $G$.

\subsubsection{Metric connection, parallel transport and curvature}

Note that, by the cyclic properties of the trace, the Riemannian metric in this group is given by
\begin{equation}\label{metrii}
\langle x,y\rangle_g=\tau( g^{-1}x(g^{-1}y)^*)=\tau ( (gg^*)^{-1} xy^* )=\langle (gg^*)^{-1}x,y\rangle_1.
\end{equation}
Then, the explicit map $g\mapsto (gg^*)^{-1}$, which is called the {\it angular momentum operator} \cite{arnold}, enables a straightforward computation of the Levi-Civita connection $\nabla$ on $GL(N)$. Let $X,Y$ be smooth vector fields considered as maps $X,Y:GL(N)\to \mathcal A$, and denote $X_g=g^{-1}X(g)$, the translation of the field $X$ to the identity, likewise for $Y$. Then
$$
\nabla_X Y(g)=DY_g(X_g)-\frac12 g \{X_gY_g+Y_gX_g+X_g^*Y_g+Y_g^*X_g-X_gY_g^*-Y_gX_g^* \}.
$$
Indeed, it is easy to check that if $f:GL(N)\to\mathbb R$ is a smooth function, then
$$
\nabla_{fX}Y=f\nabla_XY,\mbox{ and }\nabla_X(fY)=X(f) Y+f\nabla_X Y.
$$
Moreover, it is also easy to check that $\nabla$ has no torsion. What is left, is to check the compatibility of $\nabla$ with the metric
$$
X\langle Y,Z\rangle_g=\langle \nabla_XY,Z\rangle_g+\langle Y,\nabla_XZ\rangle_g.
$$
But this is also straightforward, if we use equation (\ref{metrii}), the ciclicity of the trace, and the fact that $\tau(a^*)=a$ for any $a\in\mathcal A$.

Now, consider $v,w\in\mathcal A$, $V,W$ the left-invariant vector fields given by $V_g=gv$, $W_g=gw$, and their adjoints defined as $V^*_g=gv^*$, $W^*_g=gw^*$, and note that $[V,W]_g=g[v,w]$, where now $[v,w]$ denotes the usual commutator of matrices. Then, one obtains the following simple expression for the Levi-Civita connection:
$$
\nabla_VW(g)=\frac12 g\{ [v,w]+[v,w^*]+[w,v^*]\}.
$$
Proper formulas for the sectional curvature can be obtained from here, or can be adapted from Arnol'd's book \cite[Appendix 2]{arnold}.

\subsection{Non Riemannian case}

Consider the metric space $(G,d_p)$, where $d_p$ is the rectifiable distance induced by the left invariant metric. First, we establish the following fact:

\begin{lem}\label{completo}
The space $(G,d_p)$ is a complete metric space.
\end{lem}
\begin{proof}
First, note that when $p=2$, by Hopf-Rinow's theorem $(G,d_2)$ is complete since the manifold $G$ is geodesically complete with the $2$-metric (Theorem \ref{gilberto}). Now, we claim that $d_p$ is equivalent to $d_2$ for any $p\ge 2$, a fact that will prove the claim of the lemma. Indeed, at each tangent space of $G$ (which identifies with $\a$), the $p$-norm is equivalent to the $2$-norm with constants which depend only on the dimension of $\a$. Examining the length functionals, it follows that the metrics are equivalent, with the same constants.
\end{proof}

Since $\mathcal A$ is finite dimensional, $(G,d_p)$ is also locally compact. Thus, by Cohn-Vossen's theorem (see \cite[Theorem 2.5.28]{burago}), given $h,k\in G$, there exists a short (continuous and rectifiable) path $\gamma$ joining $h,k$.  By changing the parameter, we may assume that $\gamma$ is parametrized with constant length, thus it is a Lipschitz map. We summarize our findings in the following theorem.

\begin{teo}\label{minimalidad dimension finita}
For each pair of elements $g_0,g_1\in G$, there exists a curve $\gamma$ in $G$,  such that $\gamma(0)=g_0$ and $\gamma(1)=g_1$ which has minimal length for the $p$-norm ($p=2n$). This curve $\gamma$ is of class $C^{\infty}$, and the unique minimizer in the class of continuous rectifiable paths joining given endpoints. Moreover, $v=\gamma^{-1}\dot{\gamma}$ is the unique solution of the Euler-Lagrange equation with the given initial conditions.
\end{teo}

\subsection{Local property of solutions}
In this section we show that solutions of the Euler-Lagrange  equation  have a local minimality property for the $p$-energy functional. To this effect, we recall several results from \cite{coco} concerning the Lagrangian $ E_p(x)=\|x\|_p^p$.
For $v,y\in\mathcal A$, its   second differential
$$
(D^2 E_p)_v(x,y)=\frac{d^2}{dsdt}{ E_p}\bigg|_{s=t=0}(v+sx+ty)
$$
was computed in \cite{coco} by Mata-Lorenzo and Recht:
$$
Q_v(z)=(D^2 E_p)_v(z,z)=p\|\,|z||v|^{n-1}\|_2^2+n\|\sum_{k=0}^{n-2}\| \,|v|^{n-k-2} |z^*v+v^*z| \,|v|^k \|_2^2
$$
Let us collect some facts on this quadratic form in the following proposition, the proofs can be found in  \cite{coco}.
\begin{prop}
Fix $v\in\mathcal A$. Then $Q_v:\mathcal A\to\mathbb R_{\ge 0}$. It is strictly positive if $p=2$. For $p>2$, $z\in\mathcal A$ is a direction of degeneracy of $Q_v$ if and only if $zv^*=0=v^*z$. In that case, $E(s)= E_p(v+sz)=\|v\|_p^p+s^p\|z\|_p^p$.

\medskip

In any case, there exists $\epsilon=\epsilon(p)>0$ such that if $z\in\mathcal A$ and $\|z\|_p<\epsilon(p)$, then
$$
E(s)\ge E(0)+sE'(0)+\|z\|_p^p|s|^p
$$
for any $s\in [-1,1]$.
\end{prop}
We remark that $\epsilon$ depends  on $v$ and $p$.
By Remark \ref{equivariant}, if $v(t)$ is a solution of the Euler-Lagrange equation, there exist smooth curves of unitaries $\nu_t$, $u_t$, such that $v(t)=\nu_t u_t v(0)u^*_t$. Thus
$$
\|v(t)+sz(t)\|_p^p=\|\nu_t u_tu_tv(0)u^*_t+sz(t)\|_p^p=\|v(0)+s u_t^*\nu^*_t z(t)u_t\|_p^p,
$$
and since $\|u_t^*\nu^*_tz(t)u_t\|_p^p=\|z(t)\|_p^p$, this clearly implies that $\epsilon$ can be chosen uniformly along the solution $v(t)$.

Let $\gamma_s$ be a variation of an extremal path $\gamma$, and put $v_s=\gamma_s^{-1}\dot{\gamma_s}=v(t)+sx(s,t)$. Then the expression above reads
$$
\|v_s\|_p\ge \|v(t)\|_p+s DE_{v(t)}(x(s,t))+\|x(s,t)\|_p^p |s|^p,
$$
provided $\|x(s,t)\|_p<\epsilon(p,v(0))$.

\begin{prop} Let $\gamma(t)$, $t\in \mathbb{R}$, be a smooth curve in $G$, such that $v(t)=\gamma^{-1}(t)\dot{\gamma}(t)$ is a solution of the Euler-Lagrange equation. Pick $\epsilon=\epsilon(p,v(0))$ as above. If  $\mu(t)$, $t\in I$ is a $C^1$ curve in $G$, such that $\|\mu^{-1}(t)\dot{\mu}(t)-v(t)\|_p<\epsilon$ for $t\in I$, then
$$
{\mathcal E}_p(\gamma)\le {\mathcal E}_p(\mu).
$$
\end{prop}
\begin{proof}
Denote by $m(t)=\mu^{-1}(t)\dot{\mu}(t)$ and consider the variation of $v$ given by $v_s(t)=v(t)+s[m(t)-v(t)]$. By the result cited above,
$$
\|m(t)\|_p^p=E_p(v(t)+(m(t)-v(t)))\ge E_p(v(t))+D E_{v(t)}(m(t)-v(t)) +\|m(t)-v(t)\|_p^p,
$$
for $t\in I$. Therefore
$$
{\mathcal E}_p(\mu)=\int_I \|m(t)\|_p^p dt\ge  \int_I\|v(t)\|_p^p dt +\int_I D E_{v(t)}(m(t)-v(t)) dt
$$
$$
={\mathcal E}_p(\gamma)+\int_I D E_{v(t)}(m(t)-v(t)) dt.
$$
The second integral on the right hand side vanishes, because $v$ is a critical point of the $p$-energy functional. It follows that ${\mathcal E}_p(\mu)\ge {\mathcal E}_p(\gamma)$.
\end{proof}

\section{Classical linear Banach-Lie groups}\label{schatten}

Let $\mathcal H$ be a complex separable, infinite dimensional Hilbert space. In this section we examine the geometry of the left invariant metric in the classical linear groups
$$
G_p(\h)=\{g\in\b(\h): g-1\in\b_p(\h)\},
$$
where $\b_p(\h)$ is the $p$-Schatten ideal of $\b(\h)$, $2\le p=2n <\infty$ an even integer. The Banach-Lie algebra of $G_p(\h)$ is the ideal $\b_p(\h)$. The natural norm here is the $p$-norm
$\|x\|_p=Tr((x^*x)^n)^{\frac{1}{p}}$, with $Tr$ the (possibly infinite) trace of $\b(\h)$ given by
$$
Tr(x)=\sum_j \langle x\xi_i,\xi_i\rangle,
$$
where $\langle,\rangle$ denotes the inner product of $\h$ and $(\xi_i)_{i\ge 1}$ is any orthonormal basis of $\h$.

We consider smooth curves $\alpha:[0,1]\to \b_p(\h)$, that is, we use the topology induced by the $p$-norm. With this topology, $\b_p(\h)$ is a complete metric space.

The left  invariant metric in the tangent bundle of $G_p(\h)$ is given as follows: if $x$ is tangent at $g\in G_p(\h)$ (which means that belongs to $\b_p(\h)$), then
$$
\|x\|_g=\|g^{-1}x\|_p.
$$
Many of the computations on this example are formally similar to the ones done in the previous sections, with certain small modifications. If $g(t)$ is a smooth curve in $G_p(\h)$, again we denote by $v(t)=g^{-1}(t)\frac{d}{dt}g(t)$, and consider  the Euler-Lagrange equation (\ref{variational}) of the variational problem for the $p$-energy functional,
$$
\frac{d}{dt} v(v^*v)^{n-1}+((vv^*)^n-(v^*v)^n)=0.
$$
The Legendre transformation in this context is a map among dual spaces
$$
\b_p(\h)\ni v\mapsto w=v(v^*v)^{n-1}\in\b_q(\h),
$$
where $1/p+1/q=1$, and the Hamilton equation (\ref{euler}) is again given by
$$
\dot{w}=|w|^q-|w^*|^q.
$$

By Remark \ref{lipschitz}, the Hamilton equation has local solutions. As before, if $w$ is a solution of this equation, then $b(t)=|w(t)|^2=w^*(t)w(t)$ is a solution of (\ref{ecuacionauxiliar})
\begin{equation}\label{auxdiscreto}
\dot{b}=[b^{\alpha},K]=b^{\alpha}K-Kb^{\alpha},
\end{equation}
where $K^*=-K$ is constant, and $\alpha=q/2$.

The key to prove existence of solutions of the Hamilton equation for all $t\in\mathbb{R}$ is, as before, the invariance of the spectrum over time. Recall the auxiliary equation (\ref{ecuacionauxiliar}), adapted to this context: for $K^*=-K\in\b_q(\h)$ and $\alpha=q/2$, consider
$$
\dot{b}=[b^\alpha,K].
$$

\begin{lem}
Let $b(t)$ in $\b_q(\h)$ be a positive solution of  $\dot{b}=[b^\alpha,K]$. Then the eigenvalues of $b(t)$, and their multiplicities, do not depend on $t$.
\end{lem}
\begin{proof}
Fix $t_0$. Since $b(t_0)$ lies in $\b_q(\h)$ and is positive, the non-nil part of its spectrum can be ordered as a (possibly finite) decreasing sequence of positive numbers $\lambda_k(t_0)$. Pick one of these $\lambda_k(t_0)$ and let $C_k$ be a circle centered at $\lambda_k(t_0)$ such that no other eigenvalue of $b(t_0)$ lies inside $C_k$. By the semicontinuity property of the spectrum, there exists $r>0$ such that for $|t-t_0|<r$, the spectrum of $b(t)$ does not intersect $C_k$. For such $t$, let $E_k(t)$ be the self-adjoint projection
$$
E_k(t)=\frac{1}{2\pi i}\oint_{C_k} (z-b(t))^{-1}d z.
$$
Note that $b(t)$ and $E_k(t)$ commute, and therefore $b(t)E_k(t)$ is a positive operator.
For any $m\ge 0$, the power $(b(t)E_k(t))^m$ (equal to $b(t)^mE_k(t)$ if $m\ge 1$) has constant trace. Indeed, if $m\ge 1$,
$$
\frac{d}{d t} Tr(b(t)^mE_k(t))=Tr(\dot{b}(t)b(t)^{m-1}E_k(t))+\frac{1}{2\pi i}\oint_{C_k} Tr(b(t)^m(z-b(t))^{-1}\dot{b}(t)(z-b(t))^{-1})dz.
$$
The first term equals
$$
Tr\left([b(t)^\alpha,K]b(t)^{m-1}E_k(t)\right)=Tr\left([b(t)^{\alpha+m-1}E_k(t),K]\right)=0.
$$
And in the second term,
$$
\oint_{C_k}Tr\left(b(t)^m(z-b(t))^{-1}[b(t)^\alpha,K](z-b(t))^{-1}\right)dz
$$
$$
=\oint_{C_k}Tr\left([b(t)^{\alpha+m}(z-b(t))^{-2},K]\right) dz =0.
$$
It follows that if $f$ is a continuous function in the real line, then $Tr(f(b(t)E_k(t)))$ is constant. Therefore the spectrum of $b(t)E_k(t)$ is constant. At $t=t_0$ it consist of $0$ and $\lambda_k(t_0)$. It follows that the spectrum of $b(t)$ is locally constant, and thus constant. Moreover, take $f$ a continuous function which takes the value $\frac{1}{\lambda_k}$ in a neighbourhood on $\lambda_k$, and is zero on $\lambda_j$ for $j\ne k$. Then
$$
Tr(f(b(t)E_k(t))=Tr(E_k(t))
$$
is constant, i.e. the multiplicity of $\lambda_k$ is independent of $t$.
\end{proof}
Then we have the following result, analogous to Theorem \ref{existenciatransformed}, with essentially the same proof, which we omit. We note that local existence is guaranteed because the map $x\mapsto |x|^q$ is Lipschitz in $\b_q(\h)$, as remarked before.
\begin{teo}
In the present context, the Hamilton equation (\ref{euler})
\begin{equation}\label{eulerdiscr}
\left\{
\begin{array}{l} \dot{w}=|w|^q-|w^*|^q \\
w(0)=w_0 \end{array}
\right.
\end{equation}
 has a unique $C^1$ solution for any initial condition $w_0\in\b_q(\h)$, defined for all $t\in\mathbb{R}$,
\end{teo}

\subsection{Reversing the Legendre transform}

Let $w$ be the solution of (\ref{euler}) with $w(0)=w_0$. Let $|w_0|=\sum_{i\ge 1} \lambda_i^{\frac12} p_i$, for $\lambda_i>0$. The invariance of the spectrum implies that
$$
b(t)=|w(t)|^2=\sum_{i\ge 1} \lambda_i p_i(t),
$$
with $p_i(0)=p_i$. Since $b(t)$ is continuously differentiable in $\b_q(\h)$, and each $p_i(t)$ can be obtained as a Riesz integral of $b(t)$, it follows that $p_i(t)$ are continuously differentiable in the parameter $t$. Therefore it is apparent that finite sums  of the $p_i(t)$ are continuously differentiable maps. Though it is not clear if $p_0(t)$, the kernel projections (or equivalently, $p_0^\perp(t)=\sum_{i\ge 1}p_i(t)$) are continuously differentiable, if the spectrum of $|w_0|$ is infinite.

\begin{rem}\label{sumasuma}
With the above notations, if $w=\Omega|w|$ is the polar decomposition of $w$, then
$$
w(t)=\sum_{i\ge 1}\lambda_i^{1/2} \Omega(t)p_i(t)=R(t)+K/2.
$$
with $K^*=-K$ constant and $R(t)^*=R(t)$. Note that if we denote $\sum_{i\ge 1}p_i(t)=p_0(t)^{\perp}=1-p_0(t)$, with $p_0(t)$ the projection onto the kernel of $w(t)$, then
$$
w(t)p_0^{\perp}(t)=w(t)\quad \mbox{ and } |w(t)|p_0^{\perp}(t)=|w(t)|.
$$
Since $p_0(t)w(t)=0$, then $p_0$ commutes with $R$ and $K$, and moreover
$$
p_0R=Rp_0=Kp_0=p_0K=0.
$$
By Remark \ref{descompolar}, the reversed Legendre transform of $w$ (and candidate of solution for the Euler-Lagrange equation (\ref{variational})) is therefore
\begin{equation}\label{elve}
v(t)=\sum_{i\ge 1}\lambda_i^{\frac{1}{2(2n-1)}} \Omega(t)p_i(t).
\end{equation}
Finally, note that the derivative of $w$ equals
$$
\dot{w}(t)=|w(t)|^q-|w(t)^*|^q=\sum_{i\ge 1} \lambda_i^{\alpha}(p_i(t)-\Omega(t)p_i(t)\Omega^*(t))
$$
where $\alpha=q/2$ as before.
\end{rem}
\begin{prop}
With the above notations, $v(t)\in\b_p(\h)$ given by equation (\ref{elve}) is the Legendre anti-transform of $w$, and it is continuous in $\mathbb{R}$.
\end{prop}
\begin{proof}
That for each $t$, $v(t)$ is  the Legendre anti-transform of $w$ was established in the previous remark. Let us show that it is continuous. For each $j\ge 1$, the map $p_j(t)$ is continuous for $t\in\mathbb{R}$. Therefore
$$
w(t)p_j(t)=\lambda_j\Omega(t)p_j(t),
$$ is continuous. Then the partial sums of the series
$$
\sum_{i\ge 1}\lambda_i^{\frac{1}{2(2n-1)}}\Omega(t)p_i(t)
$$ are continuous in $t$. This series is  uniformly convergent in $\b_p(\h)$. Indeed,
if
$$
S_k=\sum_{i\ge k}\lambda_i^{\frac{1}{2(2n-1)}}\Omega(t)p_i(t),
$$
then
$$
(S_k^*S_k)^n=\left(\sum_{i\ge k}\lambda_i^{\frac{1}{2n-1}}\Omega(t)p_j(t)\Omega^*(t)\right)^n.
$$
Since $\Omega^*(t)\Omega(t)=p_0^\perp$, the projections $\Omega(t)p_j(t)\Omega^*(t)$ are pairwise orthogonal:
$$
\Omega(t)p_j(t)\Omega^*(t)\Omega(t)p_l(t)\Omega^*(t)=\Omega(t)p_j(t)p_0^\perp p_l(t)\Omega^*(t)=\Omega(t)p_j(t)p_l(t)\Omega^*(t)
$$
$$
=\delta_{j,l}\Omega(t)p_j(t)\Omega^*(t).
$$
Then
$$
(S_k^*S_k)^n=\sum_{i\ge k}\lambda_i^{\frac{n}{2n-1}}\Omega(t)p_j(t)\Omega^*(t).
$$
Thus
$$
\|S_k\|_p^p=\sum_{i\ge k}\lambda_i^{\frac{n}{2n-1}}Tr(\Omega(t)p_j(t)\Omega^*(t))=\sum_{i\ge k}\lambda_i^\alpha r_j,
$$
where $r_j$ is the constant $Tr(p_j(t))$. This series is convergent, it equals the $\alpha$-power of the $\alpha$-norm of the tail of the series $b(t)=\sum_{i\ge 1}\lambda_i p_i(t)$, which is uniformly convergent in $\b_\alpha(\h)$. Therefore $v$ is continuous.
\end{proof}

\begin{rem}
Note that if $\gamma$ is any rectifiable arc, defined as in Section \ref{minimizantes}, any minimizer of the $p$-length in this setting is also a critical point for the $p$-energy functional, with the same proof as in Proposition \ref{escritica}.

\medskip

Moreover, as in Remark \ref{esabso}, if we put $\nu=\gamma^{-1}\dot{\gamma}$, and $\omega=\nu (\nu^*\nu)^{n-1}$, then $t\mapsto \omega(t)$ is absolutely continuous, thus it is also a $C^1$ map.
\end{rem}

By combining the previous results, we obtain a characterization of minimizing arcs.

\begin{coro}
If $\gamma$ is an extremal of the $p$-energy functional (in particular, if $\gamma$ is a rectifiable minimizer for the $p$-length), then $\gamma$ is $C^1$, $v=\gamma^{-1}\dot{\gamma}$ is a solution of the Euler-Lagrange equation.
\end{coro}

\subsection{A differential equation for the spectral projections}

As before, let $b(t)=\sum\lambda_i p_i(t)$ stand for the smooth solution of the differential equation (\ref{auxdiscreto}), that is $b(t)=|w(t)|^2$, where $w(t)$ is a solution of the Euler-Lagrange equation (\ref{eulerdiscr}). Then $ 0<\lambda_i\in\mathbb R$ are the singular values of $w$, thus $\{\lambda_i\}_{i\ge 1}\in \ell^{\alpha}$, with $\alpha=q/2$. Since the spectrum of $|w(t)|$ is constant and discrete (accumulating only, eventually, at $\lambda_0=0$). Let $C_i$ be a circle centered at $\lambda_i$, with no other eigenvalue $\lambda_j$ in its interior. Then
$$
p_i=p_i(t)=\frac{1}{2\pi i}\oint_{C_i} (z-b)^{-1}dz
$$
$$
\dot{p_i}=\frac{1}{2\pi i}\oint_{C_i} (z-b)^{-1}\dot{b}(z-b)^{-1}dz=\frac{1}{2\pi i}\oint_{C_i} (z-b)^{-1}[K,b^{\alpha}](z-b)^{-1}dz=Ib^{\alpha}-b^{\alpha}I,
$$
with
$$
I=\frac{1}{2\pi i}\oint_{C_i} (z-b)^{-1}K(z-b)^{-1}dz.
$$
Now since
$$
(z-b)^{-1}=\left(z(p_0+p_0^{\perp})-\sum_{i\ge 1}\lambda_i p_i\right)^{-1}=z^{-1}+\sum_{i\ge 1} \frac{\lambda_i}{z(z-\lambda_i)}p_i,
$$
a straightforward computation shows that, for each $i\ge 1$,
\begin{equation}\label{peipunto}
\dot{p_i}=\lambda_i^{\alpha-1}\left( [K,p_i]+\sum_{l\ne i} \frac{\lambda_i^{1-\alpha}-\lambda_l^{1-\alpha}}{\lambda_i-\lambda_l}\lambda_l^{\alpha}\left\{p_iKp_l-p_lKp_i\right\}\right)
\end{equation}
{\em provided the sum over $l$ converges in ${\cal B}_q({\cal H})$}. This is the purpose of the next lemmas:

\begin{lem}\label{aco}
 Let $\alpha\in (1/2,1]$, let $\{\lambda_i\}_{i\ge 1}$ be a  decreasing sequence of strictly positive numbers. Then for each $i\ge 1$,
$$
\sup\limits_{l\ne i}\frac{\lambda_i^{1-\alpha}-\lambda_l^{1-\alpha}}{\lambda_i-\lambda_l}\lambda_l^{\alpha}\le 1.
$$
\end{lem}
\begin{proof}
For given $\alpha\in (1/2,1]$, and $a>0$, consider the real function $f_a:[0,+\infty)\to\mathbb R$ given by
$$
f_a(t)=t^{\alpha}\frac{t^{1-\alpha}-a^{1-\alpha}}{t-a}.
$$
Note that $f_a(0)=0$  and that $f_a$ is continuous (we let $f_a(a)=\lim\limits_{t\to a} f_a(t)=1-\alpha$). On the other hand, it is easy to check that $\lim\limits_{t\to +\infty}f_a(t)=1$. We will show that $f_a$ is increasing, which will prove the lemma. To this end, write
$$
f_a(t)=1-a^{1-\alpha}g(t),
$$
with $g(t)=\frac{t^{\alpha}-a^{\alpha}}{t-a}$, where $g(a)=\alpha a^{\alpha-1}$ and $g:[0,+\infty)\to \mathbb R$ is again continuous. It suffices to show that $g$ is decreasing; since
$$
g'(t)=\frac{(\alpha-1)t^{\alpha}-\alpha at^{\alpha-1}+a^{\alpha} } {(t-a)^2}
$$
we only need to check that
\begin{equation}\label{numerator}
a^{-\alpha}\{(\alpha-1)t^{\alpha}-\alpha at^{\alpha-1}+a^{\alpha}\}=(\alpha-1)(t/a)^{\alpha}-\alpha (t/a)^{\alpha-1}+1
\end{equation}
is non-positive. Consider $h(z)=(\alpha-1)z^{\alpha}-\alpha z^{\alpha-1}+1$ and note that $h(0^+)=-\infty$, $h(1)=0$. Since
$$
h'(z)=\alpha(\alpha-1)z^{1-\alpha}(1-z),
$$
it follows that $h$ has a global maximum at $z=1$, and then $h\le 0$; clearly this implies that (\ref{numerator}) is non-positive.
\end{proof}

Then, to show the convergence of (\ref{peipunto}), it suffices to show the convergence in ${\mathcal B}_q({\mathcal H})$ of the sum
$$
\sum_l Kp_l,
$$
since $\|p_iKp_l\|_q\le \|Kp_l\|_q=\|p_l K\|_q$ and similarly for the term $\|p_lKp_i\|_q$. To prove that the above sum is convergent, we recall the following result (see for instance \cite[Theorem 2.16]{simonti}):

Let $a_n,a,b\in \bh$, with $b\ge 0$. Suppose that $|a_n|\le b$ and $|a_n^*|\le b$ for all $n$, $|a|\le b,|a^*|\le b$ and $a_n\to a$ weakly. If $p<\infty$ and $b\in {\mathcal B}_p({\mathcal H})$, then $\|a-a_n\|_p\to 0$.

\begin{coro}\label{kapei}
If $K$, $\{p_i\}$ are as in Remark \ref{sumasuma}, then $\sum_{i\ge 1}Kp_i\to K p_0^{\perp}=K$ in ${\mathcal B}_q({\mathcal H})$.
\end{coro}
\begin{proof}
Apply the result above to $a_n=\sum_{i=1}^n Kp_i$, $a=K$, $b=|K|$, and then recall that $Kp_0=0$ thus $Kp_0^{\perp}=K$.
\end{proof}

\begin{coro}
For each $i\ge 1$, the expression for $\dot{p_i}$ given in (\ref{peipunto}) is convergent in $ {\mathcal B}_p({\mathcal H})$.
\end{coro}

We may use these facts to prove regularity results for $b(t)$ and $|w(t)|$. First note that since $b\in\b_\alpha(\h)$ and $\alpha<1$, then $b\in\b_1(\h)$. Moreover $\dot{b}=[b^\alpha,K]$, implies that  also $\dot{b}\in\b_1(\h)$

\begin{prop}
The map $b(t)=|w(t)|^2$ is continuously differentiable in $\b_1(\h)$, and
$$
\dot{b}(t)=\sum_{i\ge 1}\lambda_i \dot{p}_i(t)
$$
with the series convergent in $\b_1(\h)$.
\end{prop}
\begin{proof}
It follows from (\ref{peipunto}), that if we denote
$$
\gamma_{i,l}=\frac{\lambda_i^{1-\alpha}-\lambda_l^{1-\alpha}}{\lambda_i-\lambda_l}\lambda_l^{\alpha}
$$
and
$$
g_i=\sum_{l\ne i}\gamma_{i,l}p_l,
$$
then
$$
\sum_{i=1}^k \lambda_i\dot{p}_i=[K,\sum_{i=1}^k\lambda_i^\alpha p_i]+\sum_{i=1}^k  \lambda_i^\alpha p_i Kg_i-\sum_{i=1}^k \lambda_i^\alpha g_iKp_i.
$$
Note that, by Lemma \ref{aco}
$$
\|g_i\|=\sup_{l\ne i}|\gamma_{i,l}|\le 1.
$$
The first term above is a partial sum of the series $[K,\sum_{i\ge 1}\lambda_i^\alpha p_i]=[K,b^\alpha]$ which converges absolutely and uniformly in $\b_1(\h)$. The second sum is bounded by
$$
\|\sum_{i=1}^k  \lambda_i^\alpha p_i Kg_i\|_1\le \sum_{i=1}^k  \lambda_i^\alpha \|p_i\|_1\| Kg_i\|\le \sum_{i=1}^k  \lambda_i^\alpha \|p_i\|_1<\infty.
$$
The third term is dealt analogously. Then $\sum_{i\ge 1} \lambda_i\dot{p}_i$ converges absolutely and uniformly in $\b_1(\h)$. Since $b(t)=\sum_{i=1} \lambda_ip_i$ also converges absolutely and uniformly in $\b_1(\h)$, it follows that $b$ is differentiable in $\b_1(\h)$ and its derivative is $\sum_{i\ge 1} \lambda_i\dot{p}_i$, which is clearly continuous.
\end{proof}
Note that the operators $g_i$ in the above proof are positive and  belong to $\b_\alpha(\h)$. Indeed, in the notations of Lemma \ref{aco}
$$
g_i=\sum_{l\ne i}f_{\lambda_i}(\lambda_l)p_l,
$$
where $f_{\lambda_i}$ is continuous in $\mathbb{R}_{\ge 0}$ with $f_{\lambda_i}(0)=0$, and clearly
$$
f_{\lambda_i}(t)=\frac{t-t^\alpha \lambda_i^{1-\alpha}}{\lambda_i-t}\sim o(t^\alpha),
$$
i.e. $f_{\lambda_i}(\lambda_l)\sim o(\lambda_l^\alpha)$.

\bigskip

\noindent

E. Andruchow, G. Larotonda and A. Varela\\
Instituto de Ciencias, Universidad Nacional de General Sarmiento,\\
J. M. Gutierrez 1150, (B1613GSX) Los Polvorines \\
Argentina\\
and \\
Instituto Argentino de Matem\'atica, CONICET \\
Saavedra 15, 3er. piso, (C1083ACA) Buenos Aires\\
Argentina\\
e-mails: eandruch@ungs.edu.ar, glaroton@ungs.edu.ar, avarela@ungs.edu.ar\\

\medskip

L. Recht\\
Universidad Sim\'on Bol\'\i var \\
Apartado 89000, Caracas 1080A,\\
 Venezuela\\
e-mail: recht@usb.ve.

\end{document}